\documentclass[a4paper,10pt]{article}
\usepackage[utf8]{inputenc}
\usepackage{enumerate}
\usepackage{amssymb, bm}
\usepackage{amsthm}
\usepackage{amsmath}
\usepackage{tikz-cd}
\usepackage{bbold}
 \usepackage[all]{xy}
 
\newtheorem{mylem}{Lemma}
\newtheorem{myprop}{Proposition}

\newtheorem{myrem}{Remark}

\title{Generalised Lelong-Poincaré formula in complex Bott-Chern cohomology}
\author{Xiaojun Wu}
\begin{document}
\newcommand{\phit}[1]{\varphi^{#1}_{\mathbf{t}}}
\newcommand{\opiccolo}[1]{\mathrm{o}\left(\left|#1\right|\right)}
\newcommand{\tempo}{\mathbf{t}}
\newcommand{\opiccolouno}{\mathrm{o}\left(1\right)}
\def\cI{\mathcal{I}}
\def\Z{\mathbb{Z}}
\def\Q{\mathbb{Q}}  \def\C{\mathbb{C}}
 \def\R{\mathbb{R}}
 \def\N{\mathbb{N}}
  \def\H{\mathbb{H}}
  \def\P{\mathbb{P}}
 \def\rC{\mathrm{C}}
  \def\d{\partial}
 \def\dbar{{\overline{\partial}}}
\def\dzbar{{\overline{dz}}}
\def \ddbar {\partial \overline{\partial}}
\def\cB{\mathcal{B}}
\def\cD{\mathcal{D}}  \def\cO{\mathcal{O}}
\def\cbarO{\overline{\mathcal{O}}}
\def\D{\mathcal{D}}
\def\cC{\mathcal{C}}
\def\cF{\mathcal{F}}
\def \rank{\mathrm{rank}}
\def \deg{\mathrm{deg}}
\def \tot{\mathrm{Tot}}
\def \Im{\mathrm{Im}}
\def \Re{\mathrm{Re}}
\def \id{\mathrm{id}}
\def \Exp{\mathrm{exp}_0}
\bibliographystyle{plain}
\def \End{\mathrm{End}}
\def \dim{\mathrm{dim}}
\def \div{\mathrm{div}}
\def \ker{\mathrm{Ker}}
\def \rC{\mathrm{Cone}}
\newcommand{\Ub}{\mathcal{U}}
\newcommand{\dcech}{\check{\delta}}
\newcommand{\dcechg}{\delta\!\!\!\check{\delta}}
\newcommand{\lc}{\mathcal{L}}
\newcommand{\ec}{\mathcal{E}}
\def\cB{\mathcal{B}}
\def\cbarO{\overline{\mathcal{O}}}
\def\cD{\mathcal{D}}
\def\cG{\mathcal{G}}
\def\cH{\mathcal{H}}
\def \H{\mathbb{H}}
\def\D{\mathcal{D}}
\def \rC{\mathrm{Cone}}
\def \Tors{\mathrm{Tors}}
\def \id{\mathrm{id}}
\def \ch{\mathrm{ch}}
\def \Td{\mathrm{Td}}
\def \Pic{\mathrm{Pic}}
\def \loc{\mathrm{loc}}
\def \pr{\mathrm{pr}}
\def \Cone{\mathrm{Cone}}
\def \Sh{\mathrm{Sh}}
\def \Ann{\mathrm{Ann}}
\def \Hom{\mathrm{Hom}}
\def \Mod{\mathrm{Mod}}
\def \Tot{\mathrm{Tot}}
\maketitle
\begin{abstract}
 In this note, we present a topological proof of the generalized Lelong-Poincaré formula.
 More precisely, when the zero locus of a section has a pure codimension equal to the rank of a holomorphic vector bundle, the top Chern class of the vector bundle corresponds to the cycle class of the schematic zero locus of the section in complex Bott-Chern cohomology.
\end{abstract}

In this note, we provide a topological proof of the generalised Lelong-Poincaré formula (\cite[ Theorem 1.1]{An07} or \cite[ Theorem 1.3]{CGL20}).
Both proofs require some geometric measure theory to apply a certain type of support theorem.
In contrast to the proofs in \cite[ Theorem 1.1]{An07} or \cite[ Theorem 1.3]{CGL20},
our approach is simpler in that it only relies on Demailly's support theorem for normal currents.

The key observation is that the Chern class of a torsion coherent sheaf of sufficiently low degree is trivial.
Recall the following proposition from \cite[Chapter V, (6.14)]{Kob87}:
if a coherent sheaf $\cF$ has support of codimension at least 2,
its determinant line bundle $\mathrm{det}(\cF)$ of $\cF$ is trivial, implying that  $c_1(\cF)=c_1(\mathrm{det}(\cF))=0$.
The first non-trivial Chern class of a torsion coherent sheaf appears precisely at the degree corresponding to the codimension of its support, which represents a cycle class. Using the Koszul complex, we can transfer the generalized Lelong-Poincaré formula to the non-trivial Chern class of a torsion sheaf of the lowest degree. When the schematic zero locus of a section is smooth, the proof becomes a direct application of the Riemann-Roch-Grothendieck formula and can be generalized to integral Bott-Chern cohomology (see Lemma 4).

Grivaux's construction of Chern classes for torsion sheaves in \cite{Gri} relies on the resolution of singularities. In general, it is difficult to provide an explicit formula for a general torsion sheaf, though an inductive approach based on the support of the torsion sheaf can be used. In contrast to \cite[Theorem 1.1]{An07} and \cite[Theorem 1.3]{CGL20}, the topological approach does not describe the Chern classes of a vector bundle beyond its top-degree class, which corresponds to the rank of the vector bundle. Nevertheless, it captures the fact that the top Chern class of a vector bundle is its Euler class. It is also worth noting that the original proof in \cite{An07} similarly relies on the resolution of singularities.

\subsection*{Acknowledgments}\label{subsec-ack}
I would like to acknowledge with deep sorrow the passing of my PhD advisor, Jean-Pierre Demailly whose guidance and discussions, particularly on the meaning of torsion sheaves, greatly influenced this research.
I express my gratitude to Professor Song Yang and Professor Xiangdong Yang for bringing attention to this question and contributing valuable insights. Additionally, I extend my thanks to the "Excellence Fellowships for Young Researchers" provided by Université Côte d’Azur for their support during this research endeavor. I would also like to acknowledge my postdoctoral advisor, Professor Laurent Stolovitch, for fostering a congenial and conducive work environment.
\section{Preliminaries}
\subsection{Rational/Complex Bott-Chern cohomology}
The complex Bott-Chern cohomology group is defined as the hypercohomology group
$$H^{p,q}_{BC}(X,\C)=\H^{p+q}(X, \cB^\bullet_{p,q, \C})$$
of the complex Bott-Chern complex
\begin{equation}
\cB^{\bullet}_{p,q,\C}: \C \xrightarrow{\Delta} \cO \oplus \cbarO  \to \Omega^1 \oplus \overline{\Omega^1} \to \cdots \to \Omega^{p-1} \oplus \overline{\Omega^{p-1}} \to \overline{\Omega^{p}} \to \cdots \to \overline{\Omega^{q-1}} \to 0
\end{equation}
where  $\Delta$ is multiplication by 1 for the first component and multiplication by -1 for the second component.
This definition coincides with the usual definition (i.e. $d-$closed $(p,q)-$forms modulo $i \d \dbar-$exact forms) as shown in \cite[Section 12, Chap. VI]{agbook}.

The integral Bott-Chern cohomology can be defined by replacing the locally constant sheaf $\C$ with $\Z(p)=(2 \pi \sqrt{-1})^p \Z$ (cf. \cite{Sch}).
The natural map from $\Z(p)$ to $\C$ induces a corresponding natural map between rational Bott-Chern cohomology and complex Bott-Chern cohomology.

\subsection{Grothendieck group}
To start with, we need the following lemma on Grothendieck group which is a variant of \cite[Proposition 7.8 in Arxiv version]{Gri07}.
For the convenience of the readers, we provide the detailed proof.
\begin{mylem}
  Let $X$ be
a complex compact manifold and $D$ a reduced subvariety. The
irreducible components of $D$ are denoted by $D_1, \cdots , D_N$, and $D_{ij} = D_i \cap D_j$.
Assume that $D_i$ are smooth and intersect transversally.
Consider the
canonical map
$$\oplus_i G_{D_i} (X) \to G_D (X)$$
where $G_{D_i}(X)$ is the Grothendieck group of coherent sheaves on $X$ supported in $D_i$.
Then there exists an exact sequence
$$\oplus_{ij}G_{D_{ij}}(X) \to \oplus_i G_{D_i} (X) \to G_D (X) \to 0.$$
\end{mylem}
\begin{proof}
We proceed by induction on
the number $N$ of the irreducible components of $D$.
We recall that we have isomorphisms of Grothendieck groups $G(D_i) \simeq G_{D_i}(X)$ induced by taking direct images (cf. e.g. \cite[Lemma 1]{Wu23}).
In the following, we will identify the objects via these isomorphisms.
Let $D'$ be the reduced variety whose irreducible components are $D_1, \cdots , D_{N-1}$. We have
a complex
$$ G(D' \cap D_N ) \to G(D') \oplus G(D_N ) \xrightarrow{\pi} G(D) \to 0$$
where the first map is given by $\alpha \mapsto (\alpha, -\alpha)$. Let us verify that this complex is exact. Consider the
map
$$\psi: \Z[\mathrm{ coh}(D) ] \to  G(D') \oplus G(D_N ) / G(D' \cap D_N )$$
defined by
$$\psi(\cF) = [i^*_{D'} \cF] + [\cI_{D'}\cF]$$
where $\Z[\mathrm{ coh}(D) ]$ is the free abelian group generated by coherent sheaves on $D$, $i^*_{D'}: D' \to D$ is the natrual inclusion, and $\cI_{D'}$ is the ideal sheaf defined by $D'$.
The symbol $[\bullet]$ denotes the class to which the element enclosed within the brackets belongs.
Note that $\cI_{D'}$ is a sheaf of $\cO_{D_N}-$modules, namely the sheaf of ideals of $D' \cap D_N$ in $D_N$ extended to $D$ by zero under the transversal intersection condition.
Let us show that $\psi$ can be defined passing to Grothendieck group.

Consider an exact sequence
$$0 \to  \cF \to \cG \to  \cH \to 0$$
of coherent sheaves on $D$. Let us define
the sheaf $\cD$ by the exact sequence
$$0 \to \cD \to  i^*_{D'}\cF \to i^*_{D'}\cG \to i^*_{D'} \cH \to 0.$$
Note that $\cD$ is a sheaf of $\cO_{D'}-$modules with support in $D' \cap D_N$ , and
$$[i^*_{D'} \cG]- [i^*_{D'} \cF]- [i^*_{D'} \cH] = -[\cD ]$$
in the Grothendieck group $G(D')$. Let us consider the following exact sequence of complexes:
$$\begin{tikzcd}
0 \arrow[r] & \cI_{D'}\cF \arrow[r] \arrow[d] & \cF \arrow[r] \arrow[d] & i^*_{D'} \cF \arrow[r] \arrow[d] & 0 \\
0 \arrow[r] & \cI_{D'}\cG \arrow[r] \arrow[d] & \cG \arrow[r] \arrow[d] & i^*_{D'} \cG \arrow[r] \arrow[d] & 0 \\
0 \arrow[r] & \cI_{D'}\cH \arrow[r]           & \cH \arrow[r]           & i^*_{D'} \cH \arrow[r]           & 0
\end{tikzcd}$$
Let $\cC$ be the first column of the diagram above, which is a complex of $\cO_{D_N}-$modules.
If we denote by $H^k(\cC)$ for $0 \leq k \leq 2$, the cohomology sheaves of $\cC$, we have the long exact sequence
$$0\to H^0(\cC)\to 0\to \cD \to H^1(\cC)\to 0\to 0\to H^2(\cC) \to 0.$$
Since $H^1(\cC)$ is a sheaf of $\cO_{D_N}-$
modules,
$\cD$ is also a sheaf of $\cO_{D_N}-$modules. Therefore, $\cD$ is a sheaf of
$\cO_{D' \cap D_N}-$modules under the transversal intersection condition.
In $G(D_N )$, we have
$$[\cI_{D'} \cF] - [\cI_{D'}\cG] + [\cI_{D'}\cH ] = [H^0(\cC)] - [H^1(\cC)] + [H^2(\cC)] = -[\cD ]. $$
Thus
$$\psi(\cF) - \psi(\cG) +
\psi(\cH) = ([\cD ], -[\cD ]) = 0$$
in the quotient.

If $\cF$ belongs to $G(D)$, then $[\cF] = [i^*_{D'} \cF] + [\cI_{D'} \cF]$ in $G(D)$. This means that
$\pi \circ \psi = id$.
Now consider
$\cH$ in $G(D')$ and $\cG$ in $G(D_N )$. Then
$$
\psi(\pi(\cH, \cG)) = ([i^*_{
D'} \cH] + [i^*_{
D'} \cG]) \oplus ([\cI_{D'} \cH] + [\cI_{D'} \cG]) = ([\cH] + [i^*_{D' \cap D_N} \cG
]) \oplus [\cI_{D' \cap D_N} \cG].$$
Remark that $[\cI_{D' \cap D_N}
\cG] = [\cG] - [i^*_{D' \cap D_N} \cG
]$ in $G(D_N
)$. Thus
$$ ([\cH] + [i^*_{D' \cap D_N} \cG
]) \oplus [\cI_{D' \cap D_N} \cG]=
[\cH] \oplus [\cG]$$ modulo $G(D' \cap D_N )$, so that $\psi \circ \pi= id$. This proves the exactness.

We can now use the induction hypothesis with $D'$. We obtain the following diagram, where the columns
as well as the first line are exact:
$$\begin{tikzcd}
0                                                & 0                                                   &                &   \\
G(D' \cap D_N) \arrow[r, "r"] \arrow[u]          & G(D') \oplus G(D_N) \arrow[r, "u"] \arrow[u]        & G(D) \arrow[r] & 0 \\
\oplus_{i <N} G(D_{iN}) \arrow[r] \arrow[u, "q"] & \oplus_{i <N} G(D_{i}) \oplus G(D_N) \arrow[u, "p"] &                &   \\
                                                 & \oplus_{i <j<N} G(D_{ij}) \arrow[u, "t"]            &                &
\end{tikzcd}$$
The map $
\oplus G(D_i) \xrightarrow{\pi} G(D)$ is onto. Let $\alpha$ be an element of $
\oplus G(D_i)$ such that $\pi(\alpha) = 0$.
Then $u(p(\alpha)) = 0$, so that there exists $\beta$ such that $r(\beta) = p(\alpha)$.
There exists $\gamma$ such that $q(\gamma) = \beta$. Then
$p(\alpha - s(\gamma)) = p(\alpha) - r(q(\gamma)) = 0$. So there exists $\delta$ such that $\alpha = s(\gamma) + t(\delta)$.
It follows that $\alpha$ is in the
image of $\oplus_{i <j} G(D_{ij})$ under $s+t$.
\end{proof}
Next, we recall the definition of Chern classes of bounded complexes.
By \cite[Lemma 13.28.2]{St}, there is a canonical isomorphism between the Grothendieck group of coherent sheaves on $X$ denoted by $K_0(X)$, and the Grothendieck group of the derived category of bounded complexes of $\mathcal{O}_X-$modules
which have coherent cohomology.
More precisely, the isomorphism is induced by the map
$$\cF^\bullet \in D^b_{coh}(X) \to \sum_i (-1)^i [H^i(\cF^\bullet)] \in K_0(X).$$
Note that, by definition, $H^i(\cF^\bullet)$ is a coherent sheaf on $X$.

By the axiomatic approach of Grivaux in \cite{Gri}, the rational Bott–Chern characteristic classes extend to a group morphism between the Grothendieck group of coherent sheaves on $X$ and
$H^{(=)}_{BC}(X, \Q)^{\times}$ (with the  multiplication structure).
More precisely,
let $\cF^\bullet$ be a bounded complexes of $\mathcal{O}_X-$modules
which have coherent cohomology.
Define
$$ch(\cF^\bullet):=\sum_i (-1)^i ch(H^i(\cF^\bullet))^{}.$$
When
$$ \cF_1^\bullet \to \cF_2^\bullet \to \cF_3^\bullet \to \cF_1^\bullet[1]$$
is a distinguished triangle in the derived category of bounded complex of $\mathcal{O}_X-$modules
which have coherent cohomology,
we have that
$$ch(\cF_2^\bullet)=ch(\cF_1^\bullet)+ch(\cF_3^\bullet) $$
from the (bounded) long exact sequence
$$\cdots \to H^j( \cF_1^\bullet) \to H^j(\cF_2^\bullet) \to H^j( \cF_3^\bullet) \to H^{j+1}(\cF_1^\bullet) \to \cdots.$$
In particular, we can define the rational Bott–Chern characteristic classes of any bounded complex of $\mathcal{O}_X-$modules
which have coherent cohomology.

Let $\cF$ be a coherent sheaf over a compact complex manifold $X$.
Grivaux's approach in \cite{Gri} defines $ch(\cF)$ in rational or complex Bott-Chern cohomology with supported in the support of $\cF$.
This observation is unnecessary to prove the main result of this note.
However, we expert to prove the generalised Lelong-Poincaré formula in rational Bott-Chern cohomology combining with a better understanding of cohomology with support.
In this note, aside from this part, we will focus on the complex Bott-Chern cohomology group.

Let us give the proof of this observation.
It is enough to consider the torsion sheaves.

If the support is smooth, it is direct consequence of the Riemann-Roch-Grothendieck formula.
Denote $Z$ as the support of $\cF$.

In the case where the support is smooth and $\cF$ is general,
by devissage, we may assume that $\cF=i_* \cG$ for some coherent sheaf on $Z$.
By the Riemann-Roch-Grothendieck formula in rational or complex Bott-Chern cohomology (cf. \cite{BSW23}, \cite{Wu23}),
$$i_*(ch(\cG)Td(Z))=ch(\cF)Td(X).$$
Thus
$ch(\cF)$ is in rational or complex Bott-Chern cohomology with supported in the support of $\cF$.

If the support is an SNC divisor (more generally when irreducible components are smooth with transversal intersection), using the exact sequence
$$\oplus_{i<j} G_{Z_i \cap Z_j}(X) \to \oplus G_{Z_i}(X) \to G_{Z}(X) \to 0$$
where $Z_i$ are irreducible components of $Z$,
we can reduce to the smooth submanifold case by Lemma 1.

If the support is singular, using embeded resolution of singularity (cf. e.g. \cite[Theorem 2.0.2]{Wlo08}), the preimage of $Z$ under $\pi: \tilde{X} \to X$ of some chosen modification of $X$ (which can be obtained by a composition of blow-ups of smooth centers and is isomorphic outside the singular locus of $Z$) satisfies that the irreducible components are smooth with transversal intersection.
The natural map
$$\cF \to R \pi_* L \pi^* \cF $$
in the derived category of bounded complexes of sheaves with coherent cohomology
induces a distinguished triangle
$$\cF \to R \pi_* L \pi^* \cF  \to \cF' \to  \cF[1]$$
where $\cF'$ has support of higher codimension than $Z$.
Note that the natural map is the identity map outside the centers of the blow-up (which can be chosen to be a proper closed analtyic set of support of $\cF$) corresponding to identity in
$$Hom(\cF,R \pi_* L \pi^* \cF )=Hom ( L \pi^* \cF,  L \pi^* \cF).$$
The conclusion follows from induction on the codimension of the support and the Riemann-Roch-Grothendieck formula
$$\pi_*(ch( L \pi^* \cF ) Td(\tilde{X}))=ch(R \pi_* L \pi^* \cF )Td(X).$$

\section{Main results}
Using Grivaux's construction, we have the following result on the representatives of Chern classes using similar ideas from the previous section.
\begin{myprop}
Let $\cF$ be a coherent sheaf over a compact complex manifold $X$.
In the complex Bott-Chern cohomology,  $ch(\cF)$ is represented by normal currents with supported in the support of $\cF$.
\end{myprop}
\begin{proof}
 We proceed by an induction on the dimension of $X$ and the dimension of the support of the coherent sheaf $\cF$.
 The induction is first on the dimension of $X$ and then on the dimension of the support of the coherent sheaf $\cF$.

 If $\cF$ is torsion-free, by \cite[Theorem 3.5]{Ros}, there exists a modification $\pi: \tilde{X} \to X$ such that $\pi^* \cF/ \Tors$ is locally free, where $\Tors$ denotes the torsion part of the corresponding coherent sheaf.
 As above, the natural map
$$\cF \to R \pi_* L \pi^* \cF $$
is generically isomorphic.
Take smooth metrics on the tangent bundle of $\tilde{X}$ and $X$ such that the Todd class can be represented by Chern-Weil forms associated with Chern connections of the corresponding metrics.
Take smooth metrics on $\pi^* \cF/ \Tors$ such that the Chern class of $\pi^* \cF/ \Tors$ can be represented by Chern-Weil forms associated with Chern connection of the corresponding metric.
The conclusion follows from induction on dimension of support and
$$\pi_*(ch( L \pi^* \cF ) Td(\tilde{X}))=ch(R \pi_* L \pi^* \cF )Td(X).$$

If $\cF$ is not a torsion sheaf,
consider the exact sequence
$$0 \to \Tors \to \cF \to \cF/ \Tors \to 0. $$
Without loss of generality and by the previous arguments, we may assume that
$\cF$ is a torsion sheaf.

Let $Z$ be the support of $\cF$.
By applying the embeded resolution of singularity, the preimage of $Z$ under $\pi: \tilde{X} \to X$ satisfies that the irreducible components are smooth with transversal intersection.
 As above, the natural map
$$\cF \to R \pi_* L \pi^* \cF $$
is generically isomorphic outside a closed analytic subset of codimension strictly greater than the codimension of $Z$.
By induction on the dimension of $Z$, the fact that the push-forward perserves normal currents, and the Riemann-Roch-Grothendieck formula,
without loss of generality, we may assume that the irreducible components of $Z$ are smooth with transversal intersection.

By devissage, we may assume that $\cF=i_* \cG$ for some coherent sheaf on a smooth submanifold $Z$.
By the Riemann-Roch-Grothendieck formula,
$$i_*(ch(\cG)Td(Z))=ch(\cF)Td(X).$$
Take smooth metrics on the tangent bundle of $Z$ and $X$ such that the Todd class can be represented by Chern-Weil forms associated with Chern connections of the corresponding metrics.
By induction of the dimension of $X$,  $ch(\cG)$ is represented by normal currents supported in the support of $\cG$.
Thus,
$ch(\cF)$ is represented by normal currents supported in the support of $\cF$.
\end{proof}

We will need the following vanishing result as a consequence of the support theorem of normal current.
\begin{mylem}
Let $\cF$ be a torsion sheaf over a connected compact complex manifold $X$.
Assume that the support of $\cF$ is of codimension at least $r+1$ in $X$.
Then for any $i \leq r$,
$$c_i(\cF)=0 \in H^{i,i}_{BC} (X, \C).$$
\end{mylem}
\begin{proof}
By Proposition 1, $c_i(\cF)$ is represented by normal currents of bidegree $(i,i)$ supported in the support of $\cF$.
These currents are trivial for $i \leq r$ by the support theorem of normal currents (cf. \cite[Chap. III, (2.11)]{agbook}).
\end{proof}
Note that the support theorem (cf. \cite[Chap. III, (2.14)]{agbook}) implies that $c_{r+1}(\cF)$ is a cycle class where $r+1$ is the codimension of the support of $\cF$ if the support is equidimensional.
\begin{mylem}
Let $\cF$ be a torsion-free coherent sheaf over a closed irreducible subvariety $Z$ in a connected compact complex manifold $X$.
Assume that $Z$ is of codimension $s$ in $X$ and $\cF$ is of rank $r$.
Let $i_Z$ be the inclusion.
Then we have
$$c_s(i_{Z*}\cF)=r\{[Z]\} \in H^{s,s}_{BC} (X, \C).$$
\end{mylem}
\begin{proof}
 Apply the embeded resolution of singularity with strict transform $\tilde{Z}$ of $Z$.
 Let $p: \tilde{Z} \to Z$ be the restriction of the resolution.
 Then the natural map
 $\cF \to p_*(p^* \cF)$
 is generically isomorphic and $R^i p_*(p^* \cF)$ (for $i \geq 1$) are torsion sheaves.
 By Lemma 2,
 $$c_s(i_{Z*} p_*(p^* \cF))=c_s(i_{Z*} \cF).$$
 Apply the Riemann-Roch-Grothendieck formula to $i_Z \circ p$.
 At the lowest degree, we have
 $$c_s(i_{Z*} p_*(p^* \cF))=i_{Z*} p_*(r)$$
 which is equal to $r\{[Z]\}$.
\end{proof}
To indicate the proof of the generalised Lelong-Poincaré formula,
let us
start with a special case.
Assume that the schematic zero locus $Z$ of some section of a vector bundle $E$ (of rank $r$) is smooth with inclusion $i: Z \to X$.
By the Riemann-Roch-Grothendieck formula in rational or complex Bott-Chern cohomology (cf. \cite{BSW23}, \cite{Wu23}),
$$i_*(Td(Z))=ch(i_* \cO_{Z})Td(X)$$
whose lowest degree gives
$$i_*(1)=c_r(\cO_X/\cI_{Z}).$$
However, $i_*(1)$ is the cycle class  in rational or complex Bott-Chern cohomology supported on $Z$.
By the Koszul resolution of $\cO_X/\cI_{Z}$ under the assumption that $Z$ is smooth, we have
$$c_r(E)=c_r(\cO_X/\cI_{Z})$$
which gives the generalised Lelong-Poincaré formula combining with the previous equality.

The proof of the general case is as follows.
\begin{myprop}
Let $E$ be a holomorphic vector bundle of rank $r$ over a connected compact complex manifold $X$ of dimension $n$ and $s : X \to E$ be a section such that $Z:=s^{-1}(0)$ (with the reduced complex space structure) is of pure
complex codimension $ r$.
Consider the analytic current $[s^{-1}(0)]$ as the sum of the irreducible components of $s^{-1}(0)$,
each multiplied with its Samuel multiplicity (cf. \cite{Bar75}).
Then we have the equality in the complex Bott-Chern cohomology
$$c_r(E)=\{[s^{-1}(0)]\}.$$
\end{myprop}
\begin{proof}
Consider the canonical
Koszul complex $K(s) = (\wedge^\bullet E^* )$ of the coherent sheaf $\cO_X /\cI_{s^{-1}(0)}$.
Here $\cI_{s^{-1}(0)}$ is the ideal sheaf defined by the schematic zero locus of $s$.
Note that the quotient sheaf $\cO_X /\cI_{s^{-1}(0)}$ may have nilpotent elements.
In general, the Koszul complex is not exact at every point.
However,  the Koszul complex is exact at the point where the (not necessarily reduced) complex space defined by $\cI_{s^{-1}(0)}$ is Cohen-Macaulay.
By \cite[Theorem 11.12, Page 75]{GPR94},
it is exact outside a proper closed analytic subset of $Z$.
By Lemma 2, we have that
$$c_r(E)=c_r(\cO_X /\cI_{s^{-1}(0)}).$$
Let $Z_i$ be the irreducible components of $Z$ with multiplicity $m_i$.
Recall that $m_i$ is the rank of $\cO_X /\cI_{s^{-1}(0)}$ (over $\cO_{Z_i}$) at a general point of $Z_i$
such that it is a smooth point of $Z_i$ disjoint with any other irreducible components and $\cO_X /\cI_{s^{-1}(0)}$ is locally free near this point.
They are the coefficients of $Z_i$ under the natural map from the Douady space of $\cO_X /\cI_{s^{-1}(0)}$ to the Barlet space.
(For more information on the natural map from the Douady space to the Barlet space, we refer to \cite[Chaptre V]{Bar75}.)

By Lemma 2, we have that
$$c_r(\cO_X /\cI_{s^{-1}(0)})= \sum_{i,j \geq 0} c_r(\cI_{Z_i}^j /(\cI_{s^{-1}(0)}\cI_{Z_i}^j+\cI_{Z_i}^{j+1})).$$
Note that the sum is finite by the assumption that $X$ is compact.
On the other hand, $\cI_{Z_i}^j /(\cI_{s^{-1}(0)}\cI_{Z_i}^j+\cI_{Z_i}^{j+1})$ can be viewed as the direct image of some coherent sheaf on $Z_i$.
In other words, they are $\cO_{Z_i}-$coherent sheaves.
By Lemma 2 and 3,
we have
$$c_r(\cO_X /\cI_{s^{-1}(0)})= \sum_{i,j \geq 0} \mathrm{rank}(\cI_{Z_i}^j /(\cI_{s^{-1}(0)}\cI_{Z_i}^j+\cI_{Z_i}^{j+1}))\{[Z_i]\}.$$
However for fixed $i$,
$$\sum_{j \geq 0} \mathrm{rank}(\cI_{Z_i}^j /(\cI_{s^{-1}(0)}\cI_{Z_i}^j+\cI_{Z_i}^{j+1}))$$
is the rank of $\cO_X /\cI_{s^{-1}(0)}$ as a $\cO_X/\cI_{Z_i}-$module at a general point of $Z_i$,
which is equal to $m_i$.
\end{proof}

The generalised Lelong-Poincaré formula holds for integral Bott-Chern cohomology if the schematic zero locus is smooth.
The following is a variant of \cite[Lemma 3.2]{Ful}.
\begin{mylem}
Let $E$ be a holomorphic vector bundle of rank $r$ over a compact complex manifold $X$.
Assume that there exists a non trivial $s \in
H^0(X, E)$ such that the (schematic) zero locus $Z(s)$ is smooth.
Then $c_r(E)$ is the cycle class associated to the zero locus in the integral Bott-Chern cohomology.
\end{mylem}
\begin{proof}
The line bundle case is a consequence of \cite[Lemma 6.47]{Wu20}.

Consider the general case.
Let $f: Y \to X$ be the splitting construction given on \cite[Page 52]{Ful} such that we have a filtration
$$E_0=0 \subset E_{r-1} \subset \cdots \subset f^* E=E_r$$
with $E_i/E_{i-1}=L_i$ line bundle quotients for any $i \geq 1$.

Let $p: \P(E) \to X$ be the projection.
 By the projection formula, we have in the integral Bott-Chern cohomology,
$$c_r(E) p_*(c_1(\cO(1))^{r-1})=p_*(p^* (c_r(E)) c_1(\cO(1))^{r-1}).$$
However, the natural map
$$H^{0,0}_{BC}(X, \Z) =\Z \to H^{0,0}_{BC}(X, \C)=\C$$
in injective.
Thus the fact that $p_*(c_1(\cO(1))^{r-1})=1 \in H^{0,0}_{BC}(X, \C)$ (by integration of the Chern curvature of any smooth metric on $\cO(1)$)
impiles that
$p_*(c_1(\cO(1))^{r-1})=1 \in H^{0,0}_{BC}(X, \Z)$.
In other words,
$$c_r(E) =p_*(p^* (c_r(E)) c_1(\cO(1))^{r-1}).$$
If $p^* (c_r(E))$
is the cycle class corresponding to
$p^{-1}({Z(s)})$, $c_r(E)$
is the cycle class corresponding to
${Z(s)}$.
The reason is as follows.
By assumption and the projecton formula,
$$p_*(p^* (c_r(E)) c_1(\cO(1))^{r-1})=p_* i_{p^{-1}({Z(s)})*}( c_1(i_{p^{-1}({Z(s)})}^* \cO(1))^{r-1})$$
where $i_{p^{-1}({Z(s)})}$ is the inclusion.
On the other hand, $i_{p^{-1}({Z(s)})}^* \cO(1)=\cO_{\P(E|_{Z(s)})}(1)$
from which we have
$$c_r(E)=i_{Z(s)*}p|_{p^{-1}(Z(s))*} ( c_1( \cO_{\P(E|_{Z(s)})}(1))^{r-1})$$
$$=i_{Z(s)*}(1)=\{Z(s)\}.$$

Since the splitting construction is given by successive projectivization of some holomorphic vector bundle,
by induction,
it is enough to show that $f^* (c_r(E))$
is the cycle class corresponding to
$f^{-1}({Z(s)})$.

The below proof follows the construction of \cite[Lemma 3.2]{Ful}.
Note that $f^* s \in H^0(Y, f^*E)$, which defines a section $s_r$ in $H^0(Y, L_r)$.
By the assumption that $Z(s)$ is smooth,
local calculation shows that
the divisor $D_r$ defined by $s_r$ is smooth.
Let $i_{D_r}$ be the inclusion map.
Then $s$ also induces a global section in $i_{D_r}^* (E_{r-1})$ whose (schematic) zero locus is isomorphic to $f^{-1}({Z(s)})$.
This global section defines a section  $s_{r-1}$ in $H^0(D_r, L_{r-1}|_{D_r})$.
Let $D_{r-1}$ be the corresponding divisor in $D_r$, which is smooth.
We can continue this procedure to define smooth cycles $D_i(1 \leq i \leq r)$ of codimension $r-i+1$ with corresponding sections $s_{i}(1 \leq i \leq r)$ (defined on $D_{i+1}$).
By construction,
the zero locus of $s_1$
is isomorphic to $f^{-1}({Z(s)})$ (i.e. $D_1=f^{-1}({Z(s)})$).

By the Whitney property, we have that
$$c_r(f^* E)= \prod_{i=1}^r c_1(L_i).$$
Thus in $H^{r,r}_{BC}(Y, \Z)$,
using the line bundle case,
$$c_r(f^* E)= \prod_{i=1}^{r-1} c_1(L_i) \cdot \{[D_r]\}=i_{D_r*}(\prod_{i=1}^{r-1} c_1(i_{D_r}^* L_i)).$$
By induction, we have that
$$c_r(f^* E)=i_{D_r*} \circ \cdots \circ i_{D_1*} (1)=\{[f^{-1}({Z(s)})]\}.$$
This finishes the proof of the lemma.
\end{proof}

\begin{myrem}
Using the natural map from the complex Bott-Chern cohomology to complex Deligne cohomology, the generalised Lelong-Poincaré formula holds in complex Deligne cohomology.
However, since complex Deligne cohomology can be represented by hypercocycles instead of smooth forms of non-pure type (i.e. not bidegree $(r,r)$ for some $r$), and there exists no support theorem for currents of non-pure type, our proof cannot apply directly to show the formula in complex Deligne cohomology.

Our proof cannot apply directly to show the formula in integral or rational Bott-Chern cohomology for the same difficulty.
\end{myrem}

\begin{myrem}
The generalised Lelong-Poincaré formula easily implies the following corollary of Demailly on Monge-Amp\`ere operator (cf. \cite[Theorem (4.5), Chap. III]{agbook}).

Let $D_1, \cdots , D_q$ be effective divisors on a compact complex manifold $X$.
Let $\cO_X(D_i)$ be the corresponding line bundles.
The currents
$[D_1] \wedge \cdots \wedge [D_q]$ are well defined as a closed positive current on $X$ as soon as
$$
\mathrm{codim}_{\C}
(D_{j_1} \cap \cdots \cap D_{j_m} ) \geq m
$$
for all choices of indices $j_1 < \cdots < j_m$ in $\{1, \cdots , q\}$.
As a corollary, the class $c_1(\cO_X(D_1)) \cdots c_1(\cO_X(D_q))$ contains a positive $(q,q)-$current as a representative under this assumption.

Define $E= \oplus_{i=1}^q \cO_X(D_i)$.
The corollary of Demailly follows from
$$c_q(E)=c_1(\cO_X(D_1)) \cdots c_1(\cO_X(D_q))$$
by applying the generalised Lelong-Poincaré formula.
\end{myrem}
\begin{myrem}
The codimension assumption is essential for the generalized Lelong-Poincaré formula, as demonstrated in Remark 2.
Take $X$ to be the blow-up of $\P^2$ with exceptional divisor $D$.
Consider the vector bundle $E=\cO_X(D) \oplus \cO_X(D)$.
Since
$$c_2(E)=c_1(\cO_X(D))^2=-1,$$
the generalized Lelong-Poincaré formula does not hold in this case.
\end{myrem}

\end{document}